\documentclass[article,english,reqno,times,11pt]{smfart}
\usepackage[utf8]{inputenc}
\usepackage{lmodern}
\usepackage[T1]{fontenc}
\usepackage{graphicx}
\usepackage{etex}
\usepackage{xcolor}
\usepackage[normalem]{ulem}
\usepackage{amsmath}
\usepackage{textcomp}
\usepackage{marvosym}
\usepackage{wasysym}
\usepackage{amssymb}
\usepackage{amsmath}
\usepackage{qtree}
\usepackage{bussproofs}
\usepackage{proof}
\usepackage{fitch}
\usepackage{cancel}
\usepackage{url}
\usepackage{smfenum}
\usepackage{smfthm}
\usepackage{setspace}
              \author{Joseph Vidal-Rosset}
	       \address{Université   de   Lorraine,   Département   de
                 philosophie,  Archives Poincaré,  UMR  7117 du  CNRS,
                 bd. Libération, 54000 Nancy - France}
               \email{joseph.vidal-rosset@univ-lorraine.fr}
\usepackage[english]{babel}
\usepackage[utf8]{inputenc}
\usepackage{lmodern}
\usepackage[T1]{fontenc}
\usepackage{graphicx}
\usepackage{etex}
\usepackage{xcolor}
\usepackage[normalem]{ulem}
\usepackage{amsmath}
\usepackage{textcomp}
\usepackage{marvosym}
\usepackage{wasysym}
\usepackage{amssymb}
\usepackage{amsmath}
\usepackage{qtree}
\usepackage{bussproofs}
\usepackage{proof}
\usepackage{fitch}
\usepackage{cancel}
\usepackage{url}
\date{\today}
\title{The Core Logic Paradox}
\begin{document}

\begin{abstract}
  This paper provides a proof  that Tennant's logical system entails a
  paradox that is called ``Core logic  paradox'', in reference to the new
  name   given    by   Tennant   to   his    intuitionistic   relevant
  logic \cite{tennant87,tennant87ndscir}.  The inconsistency theorem on which
  this paper ends  is preceded by lemmas that  prove syntactically and
  semantically  that  rule \(L\top\)  is  strongly  admissible in  the
  sequent calculus for propositional Core logic. This rule plays a key
  role in a consistency test on which Tennant’s logical system fails. 
\end{abstract}
\maketitle

\section{Two basic claims of Core logic}

The idiosyncrasy of Tennant's logical  system lies mainly in this pair
of claims:

\begin{enumerate}
\item The First Lewis  Paradox that is the sequent \(\lnot  A, A \vdash B\)
  is \emph{not} provable in Core logic, therefore the formula
  \begin{equation}
    \label{eq:1}
    \lnot A, A \nvdash B 
  \end{equation}
  is claimed as being valid in Core logic.\footnote{Tennant \cite[page
    195]{tennantcore2017} wrote \(\lnot A, A \nvDash B\), because of the
    completeness  claimed  for  Core   logic,  formula  \eqref{eq:1}  is
    justified.}
  \item Every intuitionistic logical \emph{theorem} is provable in Core
    logic, therefore
    \begin{equation}
      \label{eq:2}
      \vdash \lnot A \to (A \to B)
    \end{equation}
    is provable in Core logic. 
\end{enumerate}

The slogan  ``relevance at the  level of the turnstile''  \cite[p. 15,
p.   41, p.  121, p.   263]{tennantcore2017} explains  this surprising
feature of Tennant's logical system.

The proofs of lemmas and theorem given in the next section are based on the
following rules of sequent calculus for propositional Core logic.

\begin{table}[htbp] \label{table} 
   \caption{Rules of sequent calculus for propositional Core logic \cite{tennantcore2017}}
\begin{tabular}{cc}
\multicolumn{2}{c}{
\begin{minipage}[c]{0.96\linewidth} \small
\begin{prooftree}
       \def\fCenter{\ \vdash\ }
\AxiomC{}
\RightLabel{\scriptsize{\textit{Ax.}}}
\UnaryInf$A \fCenter A$
\end{prooftree}
\end{minipage}
}
\\
\begin{minipage}[c]{0.46\linewidth} \small
\begin{prooftree}
       \def\fCenter{\ \vdash\ }
\Axiom$ \Delta \fCenter A $
\RightLabel{\scriptsize{$L\lnot$}}
\UnaryInf$\lnot A, \Delta \fCenter$
\end{prooftree}
\end{minipage} & 
\begin{minipage}[c]{0.46\linewidth} \small
\begin{prooftree}
       \def\fCenter{\ \vdash\ }
\Axiom$A, \Delta \fCenter $
\RightLabel{\scriptsize{$R\lnot$}}
\UnaryInf$\Delta \fCenter \lnot A$
\end{prooftree}
\end{minipage}
\\
\\
\begin{minipage}[c]{0.46\linewidth} \small
\begin{prooftree} 
       \def\fCenter{\ \vdash\ }
\Axiom$ \Delta \fCenter C $
\RightLabel{\scriptsize{$L\land$}}
\UnaryInf$A \land B, \Delta \backslash\{A,B\} \fCenter C$
\noLine
\UnaryInfC{\scriptsize{where $\Delta \cap \{A, B\} \neq \emptyset $}}
\end{prooftree}
\end{minipage} & 
\begin{minipage}[c]{0.46\linewidth} \small
\begin{prooftree}
       \def\fCenter{\ \vdash\ }
\Axiom$\Delta \fCenter A $
\Axiom$\Gamma \fCenter B $
\RightLabel{\scriptsize{$R\land$}}
\BinaryInf$\Delta, \Gamma \fCenter  A \land B$
\end{prooftree}
\end{minipage}
\\
\\
\begin{minipage}[c]{0.46\linewidth} \small
\begin{prooftree}
       \def\fCenter{\ \vdash\ }
\Axiom$ A, \Delta \fCenter C/\bot $
\Axiom$ B, \Gamma \fCenter C/\bot $
\RightLabel{\scriptsize{$L\lor$}}
\BinaryInf$A \lor B, \Delta, \Gamma  \fCenter C / \bot$
\end{prooftree}
\end{minipage} & 
\begin{minipage}[c]{0.46\linewidth} \small
\begin{prooftree}
       \def\fCenter{\ \vdash\ }
\Axiom$\Delta \fCenter A $
\RightLabel{\scriptsize{$R\lor$}}
\UnaryInf$\Delta \fCenter  A \lor B$
\end{prooftree}
\begin{prooftree}
       \def\fCenter{\ \vdash\ }
\Axiom$\Delta \fCenter B $
\RightLabel{\scriptsize{$R\lor$}}
\UnaryInf$\Delta \fCenter  A \lor B$
\end{prooftree}
\end{minipage}
\\
\\
\begin{minipage}[c]{0.46\linewidth} \small
\begin{prooftree}
       \def\fCenter{\ \vdash\ }
\Axiom$\Delta \fCenter A $
\Axiom$ B, \Gamma \fCenter C $
\RightLabel{\scriptsize{$L\to $}}
\BinaryInf$A \to B, \Delta, \Gamma  \fCenter C$
\end{prooftree} 
\end{minipage} & 
\begin{minipage}[c]{0.46\linewidth} \small
\begin{prooftree}
       \def\fCenter{\ \vdash\ }
\Axiom$A, \Delta \fCenter $
\RightLabel{\scriptsize{$R\to_{(a)}$}}
\UnaryInf$\Delta \fCenter  A \to B$
\end{prooftree}
\begin{prooftree}
       \def\fCenter{\ \vdash\ }
\Axiom$\Delta \fCenter B $
\RightLabel{\scriptsize{$R\to_{(b)}$}}
\UnaryInf$\Delta \backslash \{A\} \fCenter  A \to B$
\end{prooftree}
\end{minipage}
\end{tabular} 
\end{table}

\section{Inconsistency Proof}

\begin{defi} \(\mathcal{A} = P \cup \{\lnot, \lor, \land, \to\} \cup
  \{),(\}\).  \(\mathcal{A}\) is the standard alphabet for formulas of
  propositional Core logic, composed of \(P\) as set of propositional
  variables, of the usual set of connectives that are respectively
  negation, disjunction, conjunction, conditional and a set of round
  brackets.  The symbolism of sequent calculus for this logic is also standard:
  ``\(\Gamma, \Delta\)'' stand for sets of formulas, and ``\(\vdash\)''
  is the usual symbol of syntactic inference.\footnote{In this paper ``\(\vdash\)''
    is  preferred  instead of  ``:''  used  by Tennant,  just  because
    ``\(\nvdash\)'' is also used.}
\end{defi}
\begin{defi} 
  \(\bot\) is not the absurdity constant but ``merely a punctuation device in proofs, used in order to register absurdity''
  \cite[p. 266]{tennantcore2017}.
\end{defi}
\begin{defi}\(\top\)  here  stands  for   \emph{any  theorem  of  Core
    propositional logic} and should not be read as the truth
  constant. 
\end{defi}
\begin{lemm}
  The \emph{equivalence} \(\Delta \dashv\vdash \top \land \Delta\)
    is provable ``at the level of the turnstile'' in Core logic.  
\end{lemm}

\begin{proof}
  Assume \(\top\) =  \((A \to A)\), then   
\begin{prooftree}
  \def\fCenter{\ \vdash\ }
      \AxiomC{}
      \RightLabel{\scriptsize{$Ax$}}
      \UnaryInf$A \fCenter A$
\UnaryInf$ \fCenter A\to A$
\AxiomC{}
 \RightLabel{\scriptsize{$Ax$}}
 \UnaryInf$\Delta \fCenter \Delta$
  \RightLabel{\scriptsize{$R \land$}}
\BinaryInf$ \Delta \fCenter (A \to A) \land \Delta$
\end{prooftree}
  
\begin{tabular}[tb]{ccc}
\begin{minipage}[c]{0.45\linewidth}
 \begin{prooftree}
      \def\fCenter{\ \vdash\ }
\AxiomC{}
 \RightLabel{\scriptsize{$Ax$}}
\UnaryInf$\Delta \fCenter  \Delta$
\RightLabel{\scriptsize{$L \land$}}
\UnaryInf$(A \to A) \land \Delta \fCenter \Delta $
\end{prooftree}
\end{minipage}  
&
i.e.   
&
\begin{minipage}[c]{0.45\linewidth}
\begin{prooftree}
      \def\fCenter{\ \vdash\ }
\AxiomC{}
 \RightLabel{\scriptsize{$Ax$}}
\UnaryInf$\Delta \fCenter  \Delta$
\RightLabel{\scriptsize{$Wk.$}}
\UnaryInf$A \to A, \Delta \fCenter \Delta $
\end{prooftree}
\end{minipage}
\end{tabular}
\\

Therefore, by  induction on  the structure  of these  derivations, the
equivalence \[\Delta \dashv\vdash \top \land \Delta\] is proved in Core logic.\footnote{The rule of Weakening on  the left is \emph{admissible} in Core
  logic,  but, to  prevent  the derivation  of the  negated-conclusion
  version of  the First Lewis Paradox  i.e. \(\lnot A, A  \vdash \lnot
  B\), Weakening on the left is banned in this case: 
    \begin{prooftree} 
       \def\fCenter{\ \vdash\ } 
      \Axiom$\lnot A, A \fCenter \bot$
     \RightLabel{\scriptsize{$Wk.$}}
     \UnaryInf$B , \lnot A, A \fCenter \bot$
     \end{prooftree}}
\end{proof}

\begin{lemm}
 The following couple of rules \(L\top\) is \emph{strongly admissible}\footnote{ The rule
  \begin{prooftree}
    \AxiomC{$S$}
    \UnaryInfC{$S'$}
  \end{prooftree}
is strongly admissible if and only if for every derivation of height \emph{n} of an instance of $S$ there
  is a derivation of height \(\leq\) $n$ of the corresponding instance of
  $S'$\cite[p. 1501]{dyckhoff-negri-2000}.  Rules \(L\top\)  are strongly
  admissible according to this definition  given by Dyckhoff and Negri
  \cite[p. 1501]{dyckhoff-negri-2000}, and invertible, hence the double
    inference line.} in Core
  logic: \\ 

  \begin{tabular}[c]{cc}
  \begin{minipage}[c]{0.5\linewidth}
     \begin{prooftree} 
      \def\fCenter{\ \vdash\ }
\doubleLine
\Axiom$\Delta \fCenter C$
\RightLabel{\scriptsize{$L\top$}}
\UnaryInf$\top , \Delta \fCenter C$
\end{prooftree}
\end{minipage}
\begin{minipage}[c]{0.5\linewidth}
\begin{prooftree} 
      \def\fCenter{\ \nvdash\ }
\doubleLine
\Axiom$\Delta \fCenter C$
\RightLabel{\scriptsize{$L\top$}}
\UnaryInf$\top , \Delta \fCenter C$
\end{prooftree}
  \end{minipage}
  \end{tabular}
\end{lemm}

\begin{proof}
  \(\{\top, \Delta \}\) and \(\{\Delta\}\) are equivalent (see the previous proof).
\end{proof}

\begin{proof}[First proof by contradiction]
  
1. Assume  \(\Delta \nvdash C\),   therefore  \(\nvdash \Delta \to
C\) and therefore \(\lnot(\Delta \to C)\) is satisfiable.

2. Assume  \(\top, \Delta  \vdash B\),    therefore \(\top  \land \Delta
   \vdash B\),    therefore \(\vdash  (\top \land \Delta)  \to C\)
   and therefore \((\top \land \Delta) \to C\) is satisfiable.
   
3.         \(\AxiomC{}\RightLabel{\scriptsize{$Ax$}}\UnaryInfC{$\vdash
   \top$}\DisplayProof\)  is  trivially   admissible  in  any  sequent
 calculus for propositional logic.
 
4. The  following derivation proves  in propositional Core  logic that
   the  conjunction of  assumptions 1  and 2  is \emph{inconsistent},  in other
   words \emph{unsatisfiable}: 

     \begin{prooftree}
           \def\fCenter{\ \vdash\ }
     \AxiomC{}
     \RightLabel{\scriptsize{$Ax$}}
     \UnaryInf$\Delta \fCenter \Delta$
     \AxiomC{}
     \RightLabel{\scriptsize{$Ax$}}
     \UnaryInf$\fCenter \top$
     \RightLabel{\scriptsize{$R\land$}}
     \BinaryInf$\Delta \fCenter \top \land \Delta$
     \AxiomC{}
     \RightLabel{\scriptsize{$Ax$}}
     \UnaryInf$C \fCenter C$
     \RightLabel{\scriptsize{$L\to$}}
     \BinaryInf$\Delta, (\top \land \Delta) \to C \fCenter C $  
     \RightLabel{\scriptsize{$R\to$}}
     \UnaryInf$(\top \land \Delta) \to C \fCenter \Delta \to C $  
     \RightLabel{\scriptsize{$L\lnot$}}
     \UnaryInf$ \lnot(\Delta \to C), (\top \land \Delta) \to C \fCenter $   
     \end{prooftree}
\end{proof}

\begin{proof}[Second proof by contradiction]
1. Assume \(\Delta \vdash C\),   therefore \(\vdash \Delta \to C\)
and therefore \(\Delta \to C\) is satisfiable.

2. Assume  \(\top, \Delta  \nvdash C\),    therefore \(\top  \land \Delta
   \nvdash C\),   therefore \(\nvdash  (\top \land \Delta) \to C\)
       and therefore \(\lnot((\top \land \Delta) \to C)\) is
       satisfiable.
    
3. The  following derivation proves  in propositional Core  logic that
the conjunction of assumptions 1 and 2 is \emph{inconsistent}, in other
words \emph{unsatisfiable}:

\begin{prooftree}
       \def\fCenter{\ \vdash\ }
\AxiomC{}
\RightLabel{\scriptsize{$Ax$}}
\UnaryInf$\Delta \fCenter \Delta$
\RightLabel{\scriptsize{$L\land$}}
\UnaryInf$\top \land \Delta \fCenter \Delta$
\AxiomC{}
\RightLabel{\scriptsize{$Ax$}}
\UnaryInf$C \fCenter C$
\RightLabel{\scriptsize{$L\to$}}
\BinaryInf$\top \land \Delta, \Delta \to C \fCenter C $
\RightLabel{\scriptsize{$R\to$}}
\UnaryInf$\Delta \to C \fCenter (\top  \land \Delta)  \to C $
\RightLabel{\scriptsize{$L\lnot$}}
\UnaryInf$\lnot((\top  \land \Delta)  \to C), \Delta \to C \fCenter $
\end{prooftree}
\end{proof}

 \begin{lemm}
 Rules \(L\top\) in Core logic are \emph{sound}: 
\begin{center}
\begin{tabular}[tb]{clr}
\begin{minipage}[c]{0.30\linewidth}
\begin{prooftree} 
      \def\fCenter{\ \vdash\ }
\doubleLine
\Axiom$\Delta \fCenter C$
\RightLabel{\scriptsize{$L\top$}}
\UnaryInf$\top , \Delta \fCenter C$
\end{prooftree}
\end{minipage}
 
&
\begin{minipage}[c]{0.10\linewidth}
therefore
\end{minipage}
&
\begin{minipage}[c]{0.60\linewidth}
\(\top, \Delta \vDash C  \Leftrightarrow  \Delta \vDash C\) 
\end{minipage}
 
\\
&
&
\\
\begin{minipage}[c]{0.30\linewidth}
\begin{prooftree} 
      \def\fCenter{\ \nvdash\ }
\doubleLine
\Axiom$\Delta \fCenter C$
\RightLabel{\scriptsize{$L\top$}}
\UnaryInf$\top , \Delta \fCenter C$
\end{prooftree}
\end{minipage}
 
&
\begin{minipage}[c]{0.10\linewidth}
therefore
\end{minipage}
&
\begin{minipage}[c]{0.60\linewidth}
\(\top, \Delta \nvDash C  \Leftrightarrow \Delta \nvDash C\) 
\end{minipage}
\end{tabular}
\end{center}
\end{lemm} ~\\

\begin{proof}[Equivalence   proof:  \(\top,  \Delta   \vDash   C
  \Leftrightarrow \Delta \vDash C\)]
  In Core  logic semantics, like in minimal logic's,  there are only
  three possible cases \cite[pp. 758-762]{tennant15relev} that lead all to the same conclusion:
  
1. Assume that $C$ is a  contradiction i.e. \(C = \bot\), then \(\top,
   \Delta \vDash C\)  is valid if and only if  \(\bot\) can be deduced
   from \(\Delta\), therefore \(\top, \Delta \vDash C \Leftrightarrow
   \Delta \vDash C\).
   
2. Assume  that $C$ is  a theorem i.e.   \(C = \top\),  by definition,
   \(\top  ,  \Delta   \vDash  C\)  is  reducible   to  \(\vDash  C\),
   therefore \(\top, \Delta \vDash C \Leftrightarrow \Delta \vDash C\),
   
3. Assume that $C$ is neither a contradiction nor a tautology, but only
   a satisfiable formula.  In this case, \(\top , \Delta  \vDash C\) is
   valid if and  only if $C$ is deducible  from \(\Delta\), therefore
   \(\top, \Delta \vDash C \Leftrightarrow \Delta \vDash C\).\\
   Therefore, in Core logic, \(\top, \Delta \vDash C \Leftrightarrow \Delta \vDash C\).   
 \end{proof}

 \begin{proof}[Equivalence   proof:  \(\top,  \Delta   \nvDash   C
  \Leftrightarrow \Delta \nvDash C\)]
   Again,  there are only
   three possible cases that lead all to the same conclusion:

   1.  Assume that  $C$ is  a contradiction  i.e. \(C  = \bot\).   \(\top
   \nvDash  \bot\)  is  trivial:   provided  the  consistency  of  any
   deductive system  $S$, the negation of  a theorem in $S$  is always
   unprovable in $S$, therefore it is
   provable  that \(\top ,  \Delta \nvDash  \bot\)  if and  only if  no
   contradiction  is  deducible  from \(\Delta\),  therefore  \(\top,
   \Delta \nvDash C \Leftrightarrow \Delta \nvDash C\).
   
2. Assume that  $C$ is  a theorem  i.e. \(C  = \top\),  then by  definition the
   formula \(\nvDash  C\) i.e.  \(\nvDash \top\)  is always  false and
   therefore \(\top, \Delta \nvDash  C \Leftrightarrow \Delta \nvDash
   C\).
   
3. Assume that $C$ is neither a contradiction nor a tautology, but only
   a satisfiable formula.  In this case, \(\top , \Delta  \nvDash C\) is
   provable if  and only it  is provable  that $C$ is  \emph{not} deducible
   from  \(\Delta\) because, again,  it is  trivial that  \(\top \nvDash  C\)
   when the  truth value  of $C$ can  be \(\bot\),  therefore \(\top,
   \Delta \nvDash C \Leftrightarrow \Delta \nvDash C\).

   Therefore, in Core  logic, \(\top, \Delta \nvDash  C \Leftrightarrow \Delta
   \nvDash C\). 
 \end{proof}

 \begin{lemm}
 In Core proofs, rules \(L\top\) must be written as follows
\begin{prooftree}
\AxiomC{}
\RightLabel{\scriptsize{$Ax$}}
  \UnaryInfC{$\vdash \top$}
\AxiomC{$\Delta \vdash C$}
\RightLabel{\scriptsize{$L \top$}}
\BinaryInfC{$\top , \Delta \vdash C$}
\end{prooftree}

\begin{prooftree}
\AxiomC{}
\RightLabel{\scriptsize{$Ax$}}
  \UnaryInfC{$\vdash \top$}
\AxiomC{$\Delta \nvdash C$}
\RightLabel{\scriptsize{$L \top$}}
\BinaryInfC{$\top , \Delta \nvdash C$}
\end{prooftree} 
\end{lemm}
\begin{proof}[Reason]
  Axiom \(\AxiomC{}\RightLabel{\scriptsize{$Ax$}}\UnaryInfC{$\vdash
\top$}\DisplayProof\) is  trivially admissible  in any  sequent calculus
for propositional logic.  \(\top\)  and \(\AxiomC{}\RightLabel{\scriptsize{$Ax$}}\UnaryInfC{$\vdash
\top$}\DisplayProof\) must be replaced  respectively by the mention of
a  theorem and  its  proof, symbol  \(\top\) being  not  used in  Core
logic. Last,  the double inference  line is  useless, the use  and the
reading of rule \(L\top\) being top-down. 
\end{proof}

 \begin{theo}
   Tennant's propositional Core logic is inconsistent.
 \end{theo}

 \begin{proof}
 Here is a consistency test \emph{via } rule \(L\top\) that leads to the addition of theorem
 \eqref{eq:2} on the left of sequent \eqref{eq:1}. Indeed, the following formula \begin{equation}
  \label{eq:3}
  \lnot  A \to (A \to B), \lnot A, A \nvdash B
  \end{equation}
is   provable    in   Core    logic,   \emph{via}    this   derivation
$\mathcal{D}_{1}$ where the last rule is \(L\top\): 
 \begin{prooftree}
       \def\fCenter{\ \vdash\ }
\AxiomC{}
\RightLabel{\tiny{\textit{$Ax$}}}
\UnaryInf$A \fCenter A $
\RightLabel{\tiny{\textit{$L\lnot$}}}
\UnaryInf$\lnot A, A \fCenter $
\RightLabel{\tiny{\textit{$R\to_{(a)}$}}}
\UnaryInf$\lnot A \fCenter A \to B$
\RightLabel{\tiny{\textit{$R\to_{(b)}$}}}
\UnaryInf$\fCenter \lnot A \to (A \to B)$
\AxiomC{$\lnot A, A \nvdash B$}
\RightLabel{\tiny{\textit{$L\top$}}}
\BinaryInfC{$ \lnot A \to (A \to B), \lnot A, A \nvdash B$}
\end{prooftree}
But the trouble is that the sequent
   \begin{equation}
   \label{eq:4}
   \lnot  A \to (A \to B), \lnot A, A \vdash B
   \end{equation}
   is also provable in Core logic, \emph{via} this derivation $\mathcal{D}_{2}$:
\begin{prooftree}
      \def\fCenter{\ \vdash\ }
\AxiomC{}
\RightLabel{\tiny{\textit{$Ax$}}}
\UnaryInf$A \fCenter A$
\RightLabel{\tiny{\textit{$L\lnot$}}}
\UnaryInf$\lnot A, A \fCenter$
\RightLabel{\tiny{\textit{$R\lnot$}}}
\UnaryInf$\lnot A \fCenter \lnot A$
\AxiomC{}
\RightLabel{\tiny{\textit{$Ax$}}}
\UnaryInf$A \fCenter A$
\AxiomC{}
\RightLabel{\tiny{\textit{$Ax$}}}
\UnaryInf$B \fCenter B$
\RightLabel{\tiny{\textit{$L\to$}}}
\BinaryInf$A \to B, A \fCenter B$
\RightLabel{\tiny{\textit{$L\to$}}}
\BinaryInfC{$ \lnot A \to (A \to B), \lnot A, A \vdash B$}
\end{prooftree}

  Therefore, a contradiction is derivable from propositional
  Core logic:
  \begin{prooftree}
    \AxiomC{$\mathcal{D}_{1}$}
    \noLine
    \UnaryInfC{$\vdots$}
    \noLine
    \UnaryInfC{$ \lnot A \to (A \to B), \lnot A, A \nvdash B$}
    \AxiomC{$\mathcal{D}_{2}$}
    \noLine
    \UnaryInfC{$\vdots$}
    \noLine
\UnaryInfC{$ \lnot A \to (A \to B), \lnot A, A \vdash B$}
\BinaryInfC{$\bot$}
\end{prooftree} 
\normalsize 
Consequently, according to the well-known definition of an inconsistent theory,\footnote{A theory \(\Sigma\) is
  inconsistent  if  and only  if  a  contradiction is  deducible  from
  \(\Sigma\).} Tennant's propositional Core logic is inconsistent.
 \end{proof}

\bibliography{/home/joseph/MEGA/org/reforg}
\bibliographystyle{smfplain}
\end{document}